\documentclass[11pt]{article}

\usepackage[margin=1in]{geometry}
\usepackage[T1]{fontenc}
\usepackage[utf8]{inputenc}
\usepackage{lmodern}
\usepackage{microtype}
\usepackage{amsmath,amssymb,amsthm,mathtools}
\usepackage{enumitem}
\usepackage[colorlinks=true, linkcolor=blue, citecolor=blue, urlcolor=blue]{hyperref}

\newtheorem{theorem}{Theorem}[section]
\newtheorem{lemma}[theorem]{Lemma}

\newtheorem{conjecture}[theorem]{Conjecture}
\newtheorem{problem}[theorem]{Problem}
\theoremstyle{remark}
\newtheorem{remark}[theorem]{Remark}

\newcommand{\fo}{f_o}
\newcommand{\girth}{\operatorname{girth}}

\title{On Scott's odd induced subgraph conjecture and a related problem}
\author{Bo Ning}
\date{April 22, 2026}

\begin{document}
\maketitle

\begin{abstract}
For a graph $G$, let $\fo(G)$ denote the maximum order of an induced subgraph of $G$ all of whose vertices have odd degree, and let
$\chi(G)$ denote the chromatic number of $G$. Scott (CPC, 1992) proved that $\fo(G) \ge |V(G)|/(2\chi(G))$ for every graph without isolated vertices, and conjectured that the factor $2$ can be removed. Wang and Wu (JGT, 2024) showed that this conjecture fails for bipartite graphs, but holds for line graphs. In this article, we confirm Scott's conjecture for claw-free graphs without isolated vertices, thereby strengthening the result of Wang and Wu. We also construct $K_{1,r}$-free graphs of arbitrarily large order to show that the conjecture fails for this broader class, for every integer $r \ge 4$.

Wang and Wu also asked whether $\fo(L(G)) \ge n/2$ holds for every connected regular graph $G$ of order $n \ge 3$. We show that $C_5$ is the smallest counterexample to this problem. On the positive side, we prove that if $G$ is a connected $k$-regular $C_5$-free graph on $n$ vertices with $k \ge 2$, then $\fo(L(G)) \ge n/2$.
\end{abstract}

\section{Introduction}\label{Section:Introduction}

Throughout this article, all graphs are finite and simple. An induced subgraph $H$ of a graph $G$ is called \emph{odd} (\emph{even}) if every vertex of $H$ has odd (even) degree in $H$. For a graph $G$, let $\fo(G)$ denote the maximum order of an odd induced subgraph of $G$.

The study of parity-constrained induced subgraphs goes back to a classical theorem of Gallai (see \cite[Problem~5.17]{Lovasz}), which asserts that every graph $G$ admits a partition $V(G)=V_1\cup V_2$ such that both $G[V_1]$ and $G[V_2]$ have all degrees even, and also a partition $V(G)=V_o\cup V_e$ such that $G[V_o]$ is odd and $G[V_e]$ is even. In particular, every graph on $n$ vertices contains an induced subgraph of order at least $n/2$ with all degrees even. This naturally leads to the odd analogue: how large an odd induced subgraph must a graph without isolated vertices contain?

To discuss the extremal problem, let
\[
\fo(n):=\min\{\fo(G): |V(G)|=n,\ \delta(G)\ge 1\},
\]
where $\delta(G)$ is the minimum degree of a graph
$G$. A long-standing conjecture, aptly described by Caro \cite{Caro} as ``part of the graph theory folklore,'' asserts that there exists a constant $c>0$ such that $\fo(n)\ge cn$ for every positive integer $n$. In the same paper, solving a weaker problem of Alon, Caro  \cite{Caro} proved that $\fo(n)=\Omega(\sqrt n)$. Scott \cite{Scott1992} later improved this to $\fo(n)=\Omega(n/\log n)$. A major breakthrough was obtained by Ferber and Krivelevich \cite{FerberKrivelevich}, who confirmed the folklore conjecture, by proving that one may take $c=10^{-4}$.

Once a linear bound is known in general, it becomes natural to seek sharper estimates for special graph classes. Indeed, before the work \cite{FerberKrivelevich}, a considerable body of work had already examined the problem restricted to special graph families. For trees, Caro, Krasikov, and Roditty \cite{CaroKrasikovRoditty} obtained an initial linear bound of $\fo(T)\geq \lceil|V(T)|/2\rceil$, and Radcliffe and Scott \cite{RadcliffeScott} later proved the sharp estimate $\fo(T)\ge 2\lfloor (|V(T)|+1)/3\rfloor$. Berman, Wang and Wargo \cite{BermanWangWargo} established the best possible bound $\fo(G)\ge 2|V(G)|/5$ for graphs of maximum degree at most $3$, and Hou, Yu, Li, and Liu \cite{HouYuLiLiu} proved the same sharp bound for graphs of treewidth at most $2$. For planar graphs, Rao, Hou, and Zeng \cite{RaoHouZeng} showed that $\fo(G)\ge |V(G)|/3$ for graphs of girth at least $6$, and $\fo(G)\ge 2|V(G)|/5$ for graphs of girth at least $7$, with both bounds best possible. Very recently, Ai, Guo, Gutin, Hao, and Yeo  \cite{AiEtAl} proved that $\fo(G)\ge 2|V(G)|/7$ for every graph of maximum degree at most $4$ and without isolated vertices, and that this bound is tight. These results show that the extremal behavior of $\fo(G)$ remains quite delicate even for highly structured families
of graphs.

In this paper, however, we are mainly concerned with a conjecture due to Scott \cite{Scott1992}, which arises
from a chromatic point of view. He proved that every graph $G$ without isolated vertices satisfies
$\fo(G)\ge \frac{|V(G)|}{2\chi(G)},$
and conjectured that the factor $2$ can be removed.

\begin{conjecture}[Scott's conjecture \cite{Scott1992}]
If $G$ has no isolated vertices then
\[
\fo(G)\ge \frac{|V(G)|}{\chi(G)}.
\]
\end{conjecture}

Scott's conjecture is substantially stronger than the folklore linear-bound problem on classes of bounded chromatic number. The first major progress on this conjecture was obtained by Wang and Wu \cite{WangWu}, who showed that it fails for bipartite graphs, but holds for all line graphs. Thus, the conjecture already fails in the class $\chi=2$. However, it remains open for graphs with chromatic number at least $3$. Wang and Wu  \cite{WangWu} also raised a related problem on line graphs of regular graphs, asking whether $\fo(L(G))\ge n/2$ for every connected regular graph $G$ of order $n\ge 3$.

In this article, we prove Scott's conjecture for claw-free graphs.
It is known that every line graph is claw-free \cite{Beineke}. Thus,
our result strengthens the line-graph theorem of Wang and Wu.
We point out that this is best possible with respect to forbidding stars: for every $r\ge 4$, there exist arbitrarily large $K_{1,r}$-free graphs that violate Scott's conjecture (Theorem~\ref{thm:k1rfree-counterexample}). Therefore, $K_{1,3}$-freeness is the largest induced-star-forbidding condition under which the conjecture holds in general. In fact, claw-free graphs for the more general problem had already been studied before. We define $f_k(G)$ to be
the maximum order of an induced subgraph of $G$ with all degree congruent to 1 modulo $k$, and
define $f_k(n):=\min \{f_k(G):\delta(G)\geq 1, v(G)=n\}$.
Caro \cite{Caro} proved that if $G$ is claw-free connected graph on $n$ vertices, then $f_k(G)\geq c_k(n\log n)^{1/3}$.
Scott \cite{Scott2001} improved this result to that $f_k(G)\geq (1+o(1))\sqrt{n/12}.$

Wang and Wu \cite{WangWu} also asked the following problem: Is it true that $\fo(L(G))\ge n/2$ for every connected regular graph $G$ of order $n\ge 3$?
We address this problem of Wang and Wu.
We show that $C_5$ is the smallest counterexample. We also give a short proof of
a positive result: Every connected
$k$-regular $C_5$-free graph $G$ of order $n$, with $k\ge 2$, satisfies $\fo(L(G))\ge n/2$ (Theorem~\ref{thm:regular-c5free}).
Our proof mainly uses two classical factor theorems of Petersen and of Kano, respectively.

We introduce some notations. For a graph $G$, let $\chi(G)$ denote its chromatic number; $G$ is $k$-chromatic if $\chi(G)=k$. A $2$-factor of $G$ is a collection of vertex-disjoint cycles $C_1,\dots,C_p$ such that $\bigcup_{i=1}^p V(C_i)=V(G)$. For positive integers $a,b$, an $[a,b]$-factor of $G$ is a spanning subgraph $H$ satisfying $a\le \deg_H(v)\le b$ for every vertex $v$. A cycle is even if its length is even. In this paper, $G$ is called $C_5$-free if it contains no subgraph isomorphic to $C_5$, and $K_{1,r}$-free (for $r\ge 3$) if it contains no induced subgraph isomorphic to $K_{1,r}$. We use $L(G)$ for the line graph of $G$, and for $S\subseteq V(G)$, $G[S]$ denotes the subgraph induced by $S$.

The organization of this paper is as follows. In Section \ref{Section:Introduction}, we give an introduction to
the topic. In Section \ref{Section:Clawfree}, we confirm Scott's conjecture for claw-free graphs
and show that Scott's conjecture fails for $K_{1,r}$-free graphs of arbitrarily large order, for
each integer $r\geq 4$. In Section \ref{Section:Problem-WangWu}, we point out that $C_5$ is the
smallest counterexample to an open problem of Wang and Wu, and present a positive result.
Finally, we conclude with two open problems in the last section.

\section{The claw-free case}\label{Section:Clawfree}
The purpose of this section is to resolve Scott's conjecture for claw-free graphs.

Recall that a graph is claw-free if it contains no induced copy of $K_{1,3}$. Claw-free graphs
form an important family in structural graph theory. In \cite{WangWu}, Wang and Wu
confirmed Scott's conjecture for line graphs. Unlike the proof in \cite{WangWu}, which relies mainly on case analysis, our argument is based on the
following consequence of Scott's weighted theorem (see Corollary~4 of \cite{Scott1992}). This idea
is also mentioned in \cite{Scott2001} for the generally conjecture (see the last section of \cite{Scott2001}).

\begin{lemma}[Scott \cite{Scott1992}]\label{lem:Scott-weighted}
Let $G$ be a $k$-chromatic graph without isolated vertices, and let $2\le m\le k$. Then $G$ contains an induced $m$-chromatic subgraph $H$ without isolated vertices
such that
\[
|V(H)|\ge \frac{m}{k}|V(G)|.
\]
\end{lemma}

\begin{theorem}\label{thm:clawfree}
If $G$ is claw-free and has no isolated vertices then
\[
\fo(G)\ge \frac{|V(G)|}{\chi(G)}.
\]
\end{theorem}

Before proving Theorem \ref{thm:clawfree}, we establish a useful lemma.
\begin{lemma}\label{lem:cycle}
For every integer $\ell\ge 3$,
\[
\fo(C_\ell)=2\left\lfloor \frac{\ell}{3}\right\rfloor.
\]
In particular, if $\ell\ne 5$, then
\[
\fo(C_\ell)\ge \frac{\ell}{2}.
\]
\end{lemma}

\begin{proof}
Let $S\subseteq V(C_\ell)$ and suppose that $C_\ell[S]$ is an odd induced subgraph. Since every vertex of a cycle has degree at most $2$, each vertex of $C_\ell[S]$ must have degree $1$. Therefore every component of $C_\ell[S]$ is an edge, and the chosen edges form an induced matching of $C_\ell$. Conversely, every induced matching in $C_\ell$ gives an odd induced subgraph.

Hence $\fo(C_\ell)$ is twice the maximum size of an induced matching in $C_\ell$. If $M$ is an induced matching in $C_\ell$, then each edge of $M$ forbids its two neighboring edges on the cycle, so $|M|\le \lfloor \ell/3\rfloor$. On the other hand, choosing every third edge on the cycle gives an induced matching of size $\lfloor \ell/3\rfloor$. Thus, we have $\fo(C_\ell)=2|M|=2\left\lfloor \frac{\ell}{3}\right\rfloor.$
The final assertion follows immediately, since for $\ell\ne 5$ one checks that $2\left\lfloor \frac{\ell}{3}\right\rfloor\ge \frac{\ell}{2}.$
\end{proof}

Now we are ready to prove Theorem \ref{thm:clawfree}.
\begin{proof}[Proof of Theorem \ref{thm:clawfree}]
Let $k=\chi(G)$ and $n=|V(G)|$. By Lemma~\ref{lem:Scott-weighted} with $m=2$, there exists an
induced bipartite subgraph $H$ of $G$ such that $H$ has no isolated vertices and $|V(H)|\ge \frac{2n}{k}.$
Because $H$ is an induced subgraph of the claw-free graph $G$, the graph $H$ is also claw-free.

We claim that $\Delta(H)\le 2$. Indeed, since $H$ is bipartite, the neighbors of any vertex lie in the
same part of the bipartition and are pairwise nonadjacent.
Therefore, a vertex of degree at least $3$ would be the center of an induced claw in $H$, contradicting claw-freeness.
Hence, every connected component of $H$ is either a path or a cycle. Since $H$ is bipartite, every cycle component is even.

It therefore suffices to prove the half-bound for paths and even cycles.

Let $p_t:=\fo(P_t)$, where $P_t$ is the path on $t$ vertices. Clearly,
$p_2=2$, $p_3=2$, and $p_4=2$.
Now let $t\ge 5$, and label the path $v_1v_2\dots v_t$. Take the edge $v_1v_2$, skip $v_3$, and in the remaining induced path $P_t[\{v_4,\dots,v_t\}]\cong P_{t-3}$ take a maximum odd induced subgraph. The union is again an odd induced subgraph of $P_t$, so
$p_t\ge 2+p_{t-3}.$ Indeed, using the induction hypothesis yields $p_t\geq 2+(t-3)/2=(t+1)/2\geq t/2$
for all $t\ge 2$.
Next consider an even cycle $C_{2r}$. If $r=2$, then $\fo(C_4)=2=|V(C_4)|/2$. Suppose now that $r\ge 3$.
By Lemma \ref{lem:cycle},
$\fo(C_{2r})\ge \frac{|V(C_{2r})|}{2}.$

Thus, every connected component of $H$ contains an odd induced subgraph on at least half of its vertices. Summing over the components, we obtain
$\fo(H)\ge \frac{|V(H)|}{2}.$
Since every induced subgraph of $H$ is also an induced subgraph of $G$, we conclude that
$\fo(G)\ge \fo(H)\ge \frac{|V(H)|}{2}\ge \frac{1}{2}\cdot \frac{2n}{k}=\frac{n}{k}.$
This proves the theorem.
\end{proof}

One may ask whether Scott's conjecture holds for all $K_{1,r}$-free graphs with $r\ge 3$.
We point out that the answer is affirmative only when $r=3$. Let $F$ be the bipartite graph with bipartition
\[
X=\{a,b,c,d\},\qquad Y=\{u,v,w,x,y\},
\]
and neighborhoods
\[
N(a)=\{u,x,y\},\qquad N(b)=\{v,w,x\},
\]
\[
N(c)=\{u,v,y\},\qquad N(d)=\{u,w,x\}.
\]
Equivalently, the edge set is
$E(F)=\{au,ax,ay,bv,bw,bx,cu,cv,cy,du,dw,dx\}.$

One can see our graph $F$ is a $K_{1,4}$-free counterexample with $\triangle(F)=3$.

\begin{theorem}\label{prop:F}
The graph $F$ satisfies $\Delta(F)=3$ and $\fo(F)=4$. Consequently, $F$ is $K_{1,r}$-free for every $r\ge 4$.
\end{theorem}

\begin{proof}
One can check that $\Delta(F)=3$. Thus, $F$ is $K_{1,r}$-free for every $r\ge 4$. Also, the vertex set $\{a,b,u,v\}$ induces two disjoint edges, so $\fo(F)\ge 4$. It remains to prove $\fo(F)\le 4$.

Let $S\subseteq V(F)$ be such that $F[S]$ has all degrees odd. Write $t:=|S\cap X|$. Since $F$ is bipartite, $t\in\{1,2,3,4\}$. We consider the following four cases.

\noindent\textbf{Case 1:} $t=1$. The unique selected vertex of $X$ has degree $1$ or $3$ in $F[S]$, so $|S\cap Y|\le 3$ and therefore $|S|\le 4$.

\noindent\textbf{Case 2:} $t=2$. Observe that every pair of vertices in $X$ has at
least one common neighbor in $Y$. Indeed,
\[
N(a)\cap N(b)=\{x\},\qquad N(a)\cap N(c)=\{u,y\},\qquad N(a)\cap N(d)=\{u,x\},
\]
\[
N(b)\cap N(c)=\{v\},\qquad N(b)\cap N(d)=\{w,x\},\qquad N(c)\cap N(d)=\{u\}.
\]
Any common neighbor of the two selected vertices in $X$ cannot belong to $S$, because it would have degree $2$ in $F[S]$. Hence, each selected vertex of $X$ has at most two available neighbors in $S\cap Y$. Since its degree in $F[S]$ must be positive and odd, each selected vertex of $X$ has degree exactly $1$. Thus $e(F[S])=2$. Every vertex of $S\cap Y$ also has odd degree, but its degree is at most $2$, so every vertex of $S\cap Y$ has degree $1$ as well. Therefore $|S\cap Y|=2$ and $|S|=4$.

\noindent\textbf{Case 3:} $t=3$. There are only four possibilities for $S\cap X$. A direct check yields the following four subcases.
\begin{itemize}[leftmargin=2em, itemsep=0.25em, topsep=0.25em]
\item If $S\cap X=\{a,b,c\}$, then among the vertices of $Y$ only $w$ has odd degree into $\{a,b,c\}$; selecting $w$ leaves $a$ and $c$ with degree $0$, so this case is impossible.
\item If $S\cap X=\{a,b,d\}$, then the unique choice of $S\cap Y$ that makes all three selected vertices of $X$ odd is $\{x\}$. This implies $|S|=4$.
\item If $S\cap X=\{a,c,d\}$, then the unique choice of $S\cap Y$ that makes all three selected vertices of $X$ odd is $\{u\}$. This implies $|S|=4$.
\item If $S\cap X=\{b,c,d\}$, then among the vertices of $Y$ only $y$ has odd degree into $\{b,c,d\}$; selecting $y$ leaves $b$ and $d$ with degree $0$, so this case is impossible.
\end{itemize}
Thus again $|S|\le 4$.

\noindent\textbf{Case 4:} $t=4$. If $S\cap X=X$, then a vertex of $Y$ can belong to $S$ only if it has odd degree into $X$. Now
$d_F(u)=3,$ $d_F(v)=2$, $d_F(w)=2$, $d_F(x)=3$, and $d_F(y)=2$.
So, only $u$ and $x$ are eligible. But no subset of $\{u,x\}$ makes all four vertices of $X$ odd: selecting neither gives degrees $(0,0,0,0)$ on $X$, selecting only $u$ gives $(1,0,1,1)$, selecting only $x$ gives $(1,1,0,1)$, and selecting both gives $(2,1,1,2)$. So this case is impossible.

All cases yield $|S|\le 4$, and therefore $\fo(F)=4$.
\end{proof}

Next, we present an infinite family of counterexamples. We first record the additivity of $\fo$ over disjoint unions.

\begin{lemma}\label{lem:additivity}
For any two graphs $G$ and $H$,
\[
\fo(G\cup H)=\fo(G)+\fo(H).
\]
In particular, $\fo(kG)=k\fo(G)$ for every positive integer $k$.
\end{lemma}

\begin{proof}
If $S$ induces an odd graph in $G\cup H$, then $S\cap V(G)$ induces an odd graph in $G$ and $S\cap V(H)$ induces an odd graph in $H$. Hence $|S|\le \fo(G)+\fo(H)$. Conversely, the disjoint union of a maximum odd induced subgraph of $G$ and a maximum odd induced subgraph of $H$ is an odd induced subgraph of $G\cup H$.
\end{proof}

Since $\fo(C_4)=2$, Theorem~\ref{prop:F} and Lemma~\ref{lem:additivity} immediately give
$\fo(kF\cup \ell C_4)=4k+2\ell$, where $k\ge 1$ and $\ell\ge 0$.

\begin{theorem}\label{thm:k1rfree-counterexample}
Fix $r\ge 4$. For every integer $n\ge 33$, there exists a $K_{1,r}$-free graph $G$ of order $n$ such that
\[
\fo(G)<\frac{|V(G)|}{\chi(G)}=\frac{n}{2}.
\]
In particular, Scott's conjecture fails for $K_{1,r}$-free graphs of arbitrarily large order.
\end{theorem}

\begin{proof}
For integers $k\ge 1$ and $\ell\ge 0$, define
$G_{k,\ell}:=kF\cup \ell C_4.$
Both $F$ and $C_4$ have maximum degree at most $3$, so $G_{k,\ell}$ is $K_{1,r}$-free for every $r\ge 4$. Moreover, $G_{k,\ell}$ is bipartite and has no isolated vertices, hence $\chi(G_{k,\ell})=2.$
By Lemma~\ref{lem:additivity} and Theorem~\ref{prop:F},
$\fo(G_{k,\ell})=k\fo(F)+\ell\fo(C_4)=4k+2\ell.$
On the other hand,
$|V(G_{k,\ell})|=9k+4\ell.$
Therefore,
$\fo(G_{k,\ell})=4k+2\ell<\frac{9k+4\ell}{2}=\frac{|V(G_{k,\ell})|}{2}=\frac{|V(G_{k,\ell})|}{\chi(G_{k,\ell})}.$
So every $G_{k,\ell}$ is a counterexample to Scott's conjecture.

It remains to represent every integer $n\ge 33$ in the form $n=9k+4\ell$ with $k\ge 1$ and $\ell\ge 0$. This can be done explicitly according to the residue class of $n$ modulo $4$:
\[
(k,\ell)=
\begin{cases}
\left(1,\dfrac{n-9}{4}\right), & n\equiv 1\pmod 4,\\[1.2ex]
\left(2,\dfrac{n-18}{4}\right), & n\equiv 2\pmod 4,\\[1.2ex]
\left(3,\dfrac{n-27}{4}\right), & n\equiv 3\pmod 4,\\[1.2ex]
\left(4,\dfrac{n-36}{4}\right), & n\equiv 0\pmod 4.
\end{cases}
\]
Because $n\ge 33$, the displayed value of $\ell$ is always a nonnegative integer. Hence for every $n\ge 33$ there is a graph $G_{k,\ell}$ of order $n$ with $\fo(G_{k,\ell})<n/2$.
\end{proof}

\begin{remark}
The family in Theorem~\ref{thm:k1rfree-counterexample} is disconnected. It settles the ``large-order'' form of the problem for $K_{1,r}$-free graphs with $r\ge 4$, but it does not give a connected infinite family. It is natural to ask whether there is a connected construction.
\end{remark}

\section{On a problem of Wang and Wu}\label{Section:Problem-WangWu}

\begin{problem}[Wang and Wu \cite{WangWu}]\label{Prob:WangWu}
Is it true that $\fo(L(G))\ge n/2$ for every connected regular graph $G$ of order $n\ge 3$?
\end{problem}

We point out that $C_5$ is a counterexample to this problem: as $L(C_5)=C_5$, we have
$\fo(L(C_5))=\fo(C_5)=2<\frac{5}{2}.$

\begin{theorem}\label{thm:regular-c5free}
Let $G$ be a connected $k$-regular graph of order $n$, where $k\ge 2$. If $G$ is $C_5$-free, then
\[
\fo(L(G))\ge \frac{n}{2}.
\]
\end{theorem}

\begin{proof}
We split the proof according to the parity of $k$.

\noindent\textbf{Case 1:} $k$ is even.

Let $k=2r$ with $r\ge 1$. By Petersen's $2$-factor theorem (see \cite{Mulder,Petersen}), $G$ has a spanning $2$-factor $F$. Hence $F$ is a disjoint union of cycles, $F=C_{\ell_1}\cup C_{\ell_2}\cup \cdots \cup C_{\ell_t},$
with $\ell_1+\cdots+\ell_t=n$. Since $G$ is $C_5$-free and each cycle of $F$ is a cycle in $G$, none of the $\ell_i$ equals $5$.

Now $L(F)$ is the disjoint union of the line graphs of these cycle components, so
$L(F)=L(C_{\ell_1})\cup \cdots \cup L(C_{\ell_t})\cong C_{\ell_1}\cup \cdots \cup C_{\ell_t},$
where $\ell_i\neq 5$ for $i\in [t]$.

By Lemma~\ref{lem:cycle},
$\fo(L(C_{\ell_i}))=\fo(C_{\ell_i})\ge \frac{\ell_i}{2}$, where  $1\le i\le t.$
Since $L(F)$ is a disjoint union of these components,
$\fo(L(F))=\sum_{i=1}^t \fo(L(C_{\ell_i}))\ge \frac{1}{2}\sum_{i=1}^t \ell_i=\frac{n}{2}.$
Finally, $L(F)=L(G)[E(F)]$, because adjacency of two edges depends only on whether
they share an endpoint. Therefore,
$\fo(L(G))\ge \fo(L(F))\ge \frac{n}{2}.$

\noindent\textbf{Case 2:} $k$ is odd.

If $k=3$, then the conclusion follows immediately from Corollary~3.4 of Wang and Wu \cite{WangWu}. Hence we may assume $k\ge 5$.

A theorem of Kano \cite{Kano} states that if $r$ is odd and $0<t\le 2r/3$, then every $r$-regular graph has a spanning $[t-1,t]$-factor each of whose components is regular. Taking $r=k$ and $t=3$ in Kano's theorem, we obtain a spanning $[2,3]$-factor $H$ of $G$ such that every component of $H$ is regular.

Because every component of $H$ is regular and every degree in $H$ lies in $\{2,3\}$, each component of $H$ is either a cycle or a cubic graph. Write the components as
$H=H_1\cup H_2\cup \cdots \cup H_s.$
If $H_i$ is a cycle, then it is a cycle of $G$ and hence is not a $C_5$. Therefore, Lemma~\ref{lem:cycle} already gives
$\fo(L(H_i))\ge \frac{|V(H_i)|}{2}.$
If $H_i$ is cubic, then Corollary~3.4 of Wang and Wu \cite{WangWu} yields
$\fo(L(H_i))\ge \frac{|V(H_i)|}{2}.$
Summing over all components, we have
$\fo(L(H))=\sum_{i=1}^s \fo(L(H_i))\ge \frac{1}{2}\sum_{i=1}^s |V(H_i)|=\frac{n}{2}.$
Again, $L(H)=L(G)[E(H)]$, and hence
$\fo(L(G))\ge \fo(L(H))\ge \frac{n}{2}.$
This completes the proof.
\end{proof}

The above theorem can be extended to the following by a similar method and we omit the proof.
\begin{theorem}\label{Thm:extending}
Let $G$ be a $d$-regular graph of order $n$, with $d\geq 4$. Then there exists a spanning subgraph
$H\subseteq G$ such that every component of $H$ is either a cycle or a cubic graph, and
$$\fo(L(G))\geq \frac{n}{2}-\frac{C_5(H)}{2},$$
where $C_5(H)$ is the number of components of $H$ that are isomorphic to $C_5$.
\end{theorem}

Theorem \ref{Thm:extending} has an immediate corollary. In fact, since $C_5(H)\leq \frac{n}{5}$, we have
$\fo(L(G))\geq \frac{n}{2}-\frac{1}{2}\lfloor\frac{n}{5}\rfloor\geq \frac{2n}{5}$ for $d$-regular graph $G$, where $d\geq 4$.
Consider the graph $C_5$. The above bound is a sharp answer to Problem \ref{Prob:WangWu}.

\section{Concluding remarks}
In \cite{RaoHouZeng}, Rao et al. proved that for any planar graph $G$ with girth at least 6,
$\fo(G)\geq \frac{|V(G)|}{3}$. The proof in Rao et al. \cite{RaoHouZeng} uses discharging. One may ask
what lower bound on $\fo(G)$ can be proved for planar graphs with girth at least 4 or 5. The following proposition reduces
the girth-5 planar case to the bipartite planar case.

Assume that every bipartite planar graph $H$ with no isolated vertices and $\girth(H)\ge 6$ satisfies
$\fo(H)\ge \frac{|V(H)|}{2}.$ Then we can show that every planar graph $G$ with no isolated vertices and $\girth(G)\ge 5$ satisfies
$\fo(G)\ge \frac{|V(G)|}{3}.$ Indeed, let $G$ be a planar graph with no isolated vertices and $\girth(G)\ge 5$, and let $n=|V(G)|$. Recall that Gr\"{o}tzsch's theorem \cite{Grotzsch}
states that every triangle-free planar graph is 3-colorable. Thus, $\chi(G)\leq 3$.
If $\chi(G)=2$, then $G$ is bipartite. Since $G$ is bipartite and $\girth(G)\ge 5$, we have $\girth(G)\ge 6$. By the assumed bipartite half-bound,
$\fo(G)\ge \frac{n}{2}\ge \frac{n}{3}.$
Now assume $\chi(G)=3$. Lemma \ref{lem:Scott-weighted} (applied with $(k,m)=(3,2)$) implies that $G$ has an induced bipartite subgraph $H$ with no isolated vertices and
$|V(H)|\ge \frac{2n}{3}.$
Indeed, since $H$ is an induced subgraph of a planar graph, it is planar. Also, every cycle of $H$ is a cycle of $G$, so $H$ has no $3$- or $4$-cycles; and because $H$ is bipartite, it has no odd cycle. Therefore $\girth(H)\ge 6$. Applying the assumed half-bound to $H$, we obtain
\(\fo(G)\ge \fo(H)\ge \frac{|V(H)|}{2}\ge \frac{1}{2}\cdot \frac{2n}{3}=\frac{n}{3}.\)
This proves the statement.

With the above proposition in hand, we strongly suspect the following is true:
Every bipartite planar graph $H$ with no isolated vertices and $\girth(H)\ge 6$ satisfies that
$\fo(H)\ge \frac{|V(H)|}{2}.$

On the other hand, as pointed out in \cite{AiEtAl}, Scott's conjecture for 3-chromatic graphs is still open. We strongly suspect
that the answer is positive, that is,  any 3-chromatic graph $G$ without isolated vertices satisfies that
$\fo(G)\geq \frac{|V(G)|}{3}.$ Moreover, similar as above analysis, there is a counterexample for 3-chromatic graphs
if there is a counterexample for $t$-chromatic graphs, for each $t\geq 4$.

\section*{Acknowledgements}
The author is grateful to Baoindureng Wu for sharing these problems during his visit to Xinjiang
University in August 2025.


\begin{thebibliography}{99}
\bibitem{AiEtAl}
J. Ai, Q. Guo, G. Gutin, Y. Hao, and A. Yeo, Odd induced subgraphs in graphs of maximum degree four, preprint, arXiv:2511.15489.

\bibitem{Beineke}
L.~W.~Beineke,
\newblock Characterizations of derived graphs,
\newblock {\em Journal of Combinatorial Theory} 9 (1970), 129--135.

\bibitem{BermanWangWargo}
D. M. Berman, H. Wang, and L. Wargo, Odd induced subgraphs in graphs of maximum degree three, \emph{Australasian Journal of Combinatorics} \textbf{15} (1997), 81--85.



\bibitem{Caro}
Y. Caro, On induced subgraphs with odd degrees, \emph{Discrete Mathematics} \textbf{132} (1994), 23--28.

\bibitem{CaroKrasikovRoditty}
Y. Caro, I. Krasikov, and Y. Roditty, On induced subgraphs of trees, with restricted degrees, \emph{Discrete Mathematics} \textbf{125} (1994), 101--106.



\bibitem{FerberKrivelevich}
A. Ferber and M. Krivelevich, Every graph contains a linearly sized induced subgraph with all degrees odd, \emph{Advances in Mathematics} \textbf{406} (2022), Paper No. 108534.

\bibitem{Grotzsch}
H. Gr\"{o}tzsch, Zur Theorie der diskreten Gebilde. VII. Ein Dreifarbensatz für dreikreisfreie Netze auf der Kugel. (German) Wiss. Z. Martin-Luther-Univ. Halle-Wittenberg Math.-Natur. Reihe 8 (1958/59), 109–120.

\bibitem{HouYuLiLiu}
X. Hou, L. Yu, J. Li, and B. Liu, Odd induced subgraphs in graphs with treewidth at most two, \emph{Graphs and Combinatorics} \textbf{34} (2018), 535--544.


\bibitem{Kano}
M. Kano, Factors of regular graphs, \emph{Journal of Combinatorial Theory, Series B} \textbf{41} (1986), 27--36.

\bibitem{Lovasz}
L. Lov\'asz, \emph{Combinatorial Problems and Exercises}, 2nd ed., AMS Chelsea Publishing, Providence, RI, 1993.

\bibitem{Mulder}
H.~M.~Mulder,
\newblock Julius Petersen's theory of regular graphs,
\newblock {\em Discrete Mathematics} 100 (1992), 157--175.

\bibitem{Petersen}
J.~Petersen,
\newblock Die Theorie der regul\"aren graphs,
\newblock {\em Acta Mathematica} 15 (1891), 193--220.


\bibitem{RadcliffeScott}
A. J. Radcliffe and A. D. Scott, Every tree contains a large induced subgraph with all degrees odd, \emph{Discrete Mathematics} \textbf{140} (1995), 275--279.


\bibitem{RaoHouZeng}
M. Rao, J. Hou, and Q. Zeng, Odd induced subgraphs in planar graphs with large girth, \emph{Graphs and Combinatorics} \textbf{38} (2022), Paper No. 105.



\bibitem{Scott1992}
A. D. Scott, Large induced subgraphs with all degrees odd, \emph{Combinatorics, Probability and Computing} \textbf{1} (1992), 335--349.

\bibitem{Scott2001}
A. D. Scott, On induced subgraphs with all degrees odd, \emph{Graphs Combin.} {\bf 17} (2001), 539--553.

\bibitem{WangWu}
T. Wang and B. Wu, Maximum odd induced subgraph of a graph concerning its chromatic number, \emph{Journal of Graph Theory} \textbf{107} (2024), 578--596.



\end{thebibliography}
\end{document}